\documentclass{amsart}
\usepackage{hyperref}
\usepackage[british]{babel}
\usepackage{enumerate, longtable}
\usepackage{amsmath, amscd, amsfonts, amsthm, amssymb, latexsym, comment, stmaryrd, graphicx, color}
\usepackage[all]{xypic}
\usepackage[T1]{fontenc}
\usepackage[utf8]{inputenc}

\newtheorem{thm}{Theorem}[section]
\newtheorem{lem}[thm]{Lemma}
\newtheorem{defi}[thm]{Definition}
\newtheorem{rem}[thm]{Remark}

\newtheorem{prop}[thm]{Proposition}

\newcommand{\GL}{\mathrm{GL}}

\newcommand{\calL}{\mathcal{L}}

\newcommand{\cA}{\mathcal{A}}
\newcommand{\cB}{\mathcal{B}}
\newcommand{\cC}{\mathcal{C}}
\newcommand{\cE}{\mathcal{E}}

\newcommand{\cK}{\mathcal{K}}

\newcommand{\cO}{\mathcal{O}}

\newcommand{\cS}{\mathcal{S}}

\newcommand{\cV}{\mathcal{V}}

\newcommand{\fS}{\mathfrak{S}}

\newcommand{\BB}{\mathbb{B}}
\newcommand{\CC}{\mathbb{C}}
\newcommand{\DD}{\mathbb{D}}

\newcommand{\KK}{\mathbb{K}}

\newcommand{\QQ}{\mathbb{Q}}
\newcommand{\RR}{\mathbb{R}}

\newcommand{\ZZ}{\mathbb{Z}}

\newcommand{\ignore}[1]{}

\begin{document}

\title{Simultaneous Diagonalization of Incomplete Matrices 
and Applications}
\author{Jean-Sébastien Coron \and Luca Notarnicola \and Gabor Wiese}
\address{Université du Luxembourg, 
	Maison du nombre, 6, avenue de la Fonte, L-4364 Esch-sur-Alzette, Luxembourg}
\email{
jean-sebastien.coron@uni.lu, luca.notarnicola@uni.lu, gabor.wiese@uni.lu}

\begin{abstract}
We consider the problem of 
recovering the entries of diagonal matrices 
$\{U_a\}_a$ for $a = 1,\ldots,t$ from
multiple ``incomplete'' samples $\{W_a\}_a$ of the form 
$W_a=PU_aQ$,
where $P$ and $Q$ are unknown matrices of low rank. 
We devise practical algorithms for this problem depending on the ranks 
of $P$ and $Q$.
This problem finds its motivation in cryptanalysis: 
we show how to significantly
improve previous algorithms for solving the approximate 
common divisor problem 
and breaking CLT13 cryptographic multilinear maps.
\end{abstract}

\maketitle

\section{Introduction}
\label{s:intro}

\subsection{Problem Statement}
This work considers the following computational problem
from linear algebra.

\begin{defi}
	[Problems $\mathbb{A}, \mathbb{B}, \mathbb{C}, \mathbb{D}$] 
	\label{def:prABCD}
	Let $n \geq 2, t \geq 2$ and $2 \leq p,q \leq n$ be integers.
	Let $\{U_a : 1 \leq a \leq t\}$ be
	diagonal 
	matrices in $\QQ^{n\times n}$.
	Let $\{W_a: 1 \leq a \leq t\}$ be matrices in $\QQ^{p \times q}$
	and $W_0 \in \QQ^{p\times q}$
	such that 
	$W_0$ has full rank and
	there exist
	matrices $P \in \QQ^{p \times n}$ of full rank $p$
	and $Q \in \QQ^{n \times q}$ of full rank $q$,
	such that $W_0 = P\cdot Q$
	and $W_a = 
	P \cdot U_a \cdot Q$
	for $1 \leq a \leq t$.
	We distinguish the following cases:
	\begin{align*}
	{(\mathbb{A})} \quad p=n \ \textrm{ and } \ q=n   \qquad& 
	{(\mathbb{B})} \quad p = n \ \textrm{ and } \ q < n \\
	{(\mathbb{C})} \quad p < n \ \textrm{ and } \ q = n \qquad &
	{(\mathbb{D})} \quad p < n \ \textrm{ and } \ q = p
	\end{align*}
	In each of the four cases, the problem states as follows: 
	\begin{itemize}
		\item [(1)] 
		Given the matrices $\{W_a : 0 \leq a \leq t\}$,
		compute $\{(u_{1,i},\ldots,u_{t,i}) : 1 \leq i \leq n\}$,
		where for $1 \leq a \leq t$, $u_{a,1},\ldots,u_{a,n} \in \QQ$
		are the diagonal entries of
		matrices $\{U_a : 1 \leq a \leq t\}$ as above.
		\item [(2)] 
		Determine whether the solution is unique.
	\end{itemize}
\end{defi}

Problem $\mathbb{A}$ is straightforward for any $t \geq 1$ 
by simultaneous diagonalization of $W_0^{-1}W_a=Q^{-1}U_aQ$ for every $a$.
Problems $\BB$ and $\CC$ are equivalent in view of their
symmetry in $p$ and $q$, and any algorithm for one solves the other upon transposing.
Therefore, we shall devise algorithms for
$\CC$ and $\DD$ only.
We refer to the matrices $\{W_a\}_a$ as ``incomplete'', as the 
low rank matrices $P$ and/or $Q$ ``steal'' information.
Of interest is the case when $p$ is much smaller than $n$.
We remark that Problem $\mathbb{A}$ is an underlying problem in 
previous works \cite{corper,cheon15} in cryptanalysis.

\subsection{Our Contributions}
Mainly, we provide efficient algorithms for 
Problems $\CC$ and $\DD$ of Def.~\ref{def:prABCD},
and show how to minimize the parameters
$p$ and $t$ with respect to $n$.
We further propose two concrete applications of our algorithms
in cryptography.
We believe that our algorithms are of independent interest 
and hope that more applications are to be found.
\\

\paragraph{\itshape{Algorithms for Problems $\CC$ and $\DD$}}
Our approach to Problem $\mathbb{C}$ is to use 
the invertibility of $Q$ and write 
$W_a = {P}{U}_a{Q}= 
{P}{Q} {Q}^{-1}{U}_a {Q} = {W}_0 {Z}_a$ 
with ${Z}_a = {Q}^{-1}{U}_a {Q}$, for every $1 \leq a \leq t$.
As ${W}_0$ is not invertible, we cannot recover 	
${Z}_a$ directly.
However we interpret this as a system of linear equations 
to solve for
$\{Z_a\}_a$. 
This system is, in general, underdetermined and does not yield 
the matrices $\{Z_a\}_a$ uniquely.
However, exploiting the special feature that $\{Z_a\}_a$ 
commute among each other
leads to additional linear equations.
This enables to recover $\{Z_a\}_a$ uniquely,
and simultaneous diagonalization eventually yields the 
diagonal entries of $\{U_a\}_a$.
We determine exact bounds on the parameters 
to ensure that we have at least as many linear equations as variables;
we obtain that $p$ and $t$
can be set as $\cO(\sqrt{n})$.
Our algorithm is heuristic only, but performs well in practice.

We reduce Problem $\DD$ to Problem 
$\CC$ by ``augmenting'' $Q$ with extra columns
so that it becomes invertible.
In this case, we show that $p$ can be close to $2n/3$.
We refer to Sec.~\ref{s:algo_C} and \ref{s:algo_D} for a complete
description of our algorithms and
provide the results of practical experiments in 
Sec.~\ref{s:practical}.\\

\paragraph{\itshape{Improved algorithm for 
		an approximate common divisor problem}}

Approximate common divisor problems 
have gained a lot of interest and different variants have been investigated. 
In \cite{cohnheninger}, Cohn and Heninger study generalizations
of the approximate common divisor problem via lattices. 
A simple version including only a single prime number 
is studied in \cite{galacd}.
A lattice cryptanalysis of the single-prime version is 
described in \cite{dghv_FHE}.
In this work we consider the multi-prime version 
(CRT-ACD Problem) from \cite{corper}, which is a 
factorization problem with constraints based on Chinese Remaindering.

We improve the two-step algorithm by \cite{corper}.
Namely, we remark that \cite{corper} relies on solving a certain 
instance of Problem $\mathbb{A}$. 
By solving an appropriate instance of Problem 
$\CC$ instead, we obtain a quadratic improvement in the number of input samples.
Namely, letting $n$ be the number of secret primes in the 
public modulus $M$, we can factor $M$ given only 
$\cO(\sqrt{n})$ input samples, 
whereas \cite{corper} uses $\cO(n)$.
We therefore achieve complete factorization of the public modulus
while limiting the input size drastically. \\

\paragraph{\itshape{Improved cryptanalysis of CLT13 Multilinear Maps}}

In 2013, \cite{ggh13} described the first 
construction of cryptographic multilinear maps,
and since then, many 
important applications in cryptography were found. 
A similar construction over the integers 
was described in \cite{CLT13} and a third construction based on the LWE Problem
was proposed \cite{ggh15}.
In the last years, many attacks against these constructions
appeared. The most devastating are 
the so-called ``zeroizing attack'', 
exploiting the availability of low-level encodings 
of zero.
The algorithm \cite{cheon15} recovers 
all secret parameters of \cite{CLT13} in the 
multiparty Diffie-Hellman key exchange.
Similar attacks have been described against GGH13 
and GGH15, see \cite{hujia16,cllt16}.

Our third contribution is therefore to improve the 
cryptanalysis of Cheon {\sl et al.} \cite{cheon15} against CLT13
when fewer encodings are public. 
Namely, \cite{cheon15} relies on solving some instance of
Problem $\mathbb{A}$. 
By solving instances of Problems $\CC$ or $\DD$ instead,
we can  lower the number of public encodings required for the cryptanalysis.
Specifically,
for a composite modulus $x_0$
of $n$ primes,
we obtain improved algorithms 
using only $\cO(\sqrt{n})$ encodings of zero
(compared to $n$ in \cite{cheon15}), 
or in total $4n/3$ encodings
(compared to $2n+2$ in \cite{cheon15}).
We confirm our results with practical experiments in Sec.~\ref{s:practical}.

\section{Notations and Preliminary Remarks}
\label{s:notations_remarks}

\subsection{Notation}
\label{ss:notation}
For $n \in \ZZ_{\geq 1}$, let $[n]$ be the set $\{1,\ldots,n\}$.
For a set $R$ and $r,s\in \ZZ_{\geq 1}$, we let $R^{r\times s}$ be the set of 
$r\times s$ matrices with entries in $R$.
For $A \in R^{r \times s}$ and $B \in R^{r \times s'}$, $[A|B] \in R^{r\times(s+s')}$
is the matrix obtained by concatenating the columns of $A$ and $B$. 
We let $1_n$ be the identity matrix
in dimension $n\in \ZZ_{\geq 1}$.
For a set $\cS$, its cardinality is denoted by $\#\cS$.

\subsection{Remarks about Definition \ref{def:prABCD}}
\label{ss:remarks_defABCD}
We shall make some important 
considerations about Def.~\ref{def:prABCD}.\\

(i) 
Let $\{W_a\}_a$ be as in Def.~\ref{def:prABCD}, 
$\pi \in \fS_n$ be a permutation
with associated 
matrix $A_{\pi} \in \{0,1\}^{n \times n}$
and $D$ any invertible diagonal $n \times n$ matrix.
Then $P' = PDA_{\pi}$ and $Q'=A_{\pi}^{-1}D^{-1}Q$ satisfy $W_0=P'Q'$ and 
$W_a=P'U_a'Q'$ for all $a \in [t]$, where $U_a'=A_{\pi}^{-1} U_a A_{\pi}$ is obtained 
from $U_a$ by permuting its diagonal entries via $\pi$.
Thus, $P',\{U'_a\}_a$ and $Q'$ satisfy the same problem.
For this reason, we only ask to recover the 
set $\{(u_{1,i},\ldots,u_{t,i}):1 \leq i \leq n\}$ in Def.~\ref{def:prABCD}.\\

(ii) 
If $t=1$ in Problem $\CC$, then the problem is not solvable 
because its solution is not unique.
Namely, we write $W_1 = W_0Z_1$, 
where $Z_1 = Q^{-1}U_1Q$ is diagonalizable with eigenvalues 
the diagonal entries of $U_1$.
But also, for every $v \in \ker(W_0)$ 
one has $W_1=W_0(Z_1+vw_1^T)$ for some $w_1 \in \QQ^n$.
Now, $Z_1$ and $Z_1+vw_1^T$ likely have different eigenvalues
which means that the solution is not unique.\\

(iii) 
There are cases when the problem is clearly not solvable for $p<n$.
For example, if $P = [1_p|0_{p\times (n-p)}]$ 
then for all $a$ the matrix $PU_a$ only 
involves the first $p$ diagonal entries of $U_a$ and the information on the 
remaining $n-p$ is lost.
These cases will not occur for "generic" or "random" instances of the problem.\\


(iv) 
If a matrix $W_0=PQ$ is not available as input
(we call it a "special input" here),
then one can recover ratios of diagonal entries of the matrices 
$\{U_a\}_a$, if $t \geq 3$.
Namely, defining $P'=PU_1$ and assuming that $U_1$ is invertible, 
one obtains $W_0':=P'Q=W_1$
and for $2 \leq a \leq t$, $W_a' := P'(U_aU_1^{-1})Q = W_a$. 
Running the algorithm on input $\{W'_a:0 \leq a \leq t-1\}$ reveals 
the tuples of diagonal entries
of the matrices $U_a U_1^{-1}$ for $1 \leq a \leq t-1$.
We will use this approach in Sec.~\ref{ss:improve_clt13} 
to improve the (CLT13) multilinear map 
cryptanalysis.\\

(v) 
For simplicity, we have stated Def.~\ref{def:prABCD} 
over $\QQ$. 
More generally, we can consider matrices over a field 
$\KK$ with exact linear 
algebra (e.g. solving linear systems, diagonalizing matrices, etc.).
Our algorithms apply to that case.

\section{An Algorithm for Problem $\CC$}
\label{s:algo_C}

We describe an algorithm to solve Problem $\CC$ of Def.~\ref{def:prABCD}.

\subsection{Description}
\label{ss:algo_C_description}

Consider integers $n,t \geq 2$ and $2 \leq p < n$
and an instance of Problem $\CC$.
We remark that it is enough to solve the following problem.
\begin{defi}[Problem $\CC'$]
	Let integers $n,t \geq 2$ and $2 \leq p < n$. Given
	\begin{itemize}
		\item a matrix $V \in \QQ^{p \times n}$ of rank $p$
		and a basis matrix $E \in \QQ^{n \times (n-p)}$
		of $\ker(V)$,
		\item a set of matrices $\{Y_a:a \in [t]\} \subseteq \QQ^{n \times n}$
	\end{itemize}
compute matrices $\{X_a:a \in [t]\} \subseteq \QQ^{(n-p)\times n}$,
such that the matrices $Y_a + EX_a$ for $a \in [t]$ 
commute with each other.
\end{defi} 


\begin{prop}\label{prop:Cprime_solves_C}
Let $\{W_a:0\leq a \leq t\}$ as in Problem $\mathbb{C}$.
Let $E \in \QQ^{n \times (n-p)}$ be a basis matrix
of the kernel of $W_0$.
Let 
$W_0^+ $ be a right-inverse\footnote{If $W_0$ (of full rank $p$) is defined over the complex
numbers, one can take $W_0^+=W_0^*(W_0 W_0^*)^{-1}$ where $W_0^*$
is the conjugate transpose of $W_0$, and $W_0^*=W_0^T$
over the real numbers.} of $W_0$.
Define 
$V = W_0$ 
and $Y_a = W_0^+ W_a$ for $a \in [t]$.
Assume that Problem $\mathbb{C}'$ is uniquely solvable
for the input matrices 
$V, E$ and $\{Y_a: a \in [t]\}$.
	
Then Problem $\mathbb{C}$ is uniquely 
solvable for the input matrices 
$\{W_a : 0 \leq a \leq t\}$.
Moreover, the matrix $Q$ in the 
assumption of Problem $\mathbb{C}$ 
is unique up to multiplication
 by a permutation matrix and an invertible diagonal matrix if 
at least one of the matrices 
$\{U_a\}_a$ has pairwise distinct diagonal
entries.
\end{prop}

\begin{proof}
	Write $W_0 = PQ$ and
	$W_a = P U_a Q$ 
	as in Problem $\mathbb{C}$.
	For all $a \in [t]$, we write
	$W_a = (PQ) (Q^{-1} U_a Q) = W_0 Z_a$,
	where 
	$Z_a := Q^{-1} U_a Q$.
	The matrices $\{Z_a : a\in [t]\}$ commute and are 
	simultaneously diagonalizable.
	For every $a \in [t]$, $Z_a$ can be written as
	$Z_a = Y_a + EX_a$
	for some $X_a \in \QQ^{(n-p)\times n}$
	since 
	$W_0 Y_a = W_a$.
	Since the matrices $\{Z_a\}_a$ 
	commute, 
	${V}$, ${E}$ and $\{{Y}_a\}_a$
	define a valid input for Problem $\mathbb{C}'$.
	By assumption, we can compute
	the matrices $\{X_a\}_a$ 
	by solving Problem $\mathbb{C}'$ and these
	are unique.
	From the knowledge of $\{X_a\}_a$,
	we compute $Z_a = Y_a + EX_a$
	for $a \in [t]$. Then these matrices are also unique. 
	Thus the set of tuples of eigenvalues $\{(u_{1,i},\ldots,u_{t,i}) : 1 \leq i \leq n\}$
	is unique and can be computed by simultaneous 
	diagonalization. 
	
	For the last part of the statement, 
	assume that we have matrices $P',Q'$,
	diagonal matrices $\{U'_a\}_a$,
	which are necessarily of the form 
	$U_a' = A^{-1}U_aA$ for a permutation matrix $A$, 
	such that
	$W_0 = P' Q'$ and $W_a' = P' U_a' Q'$ for every $a$.
	By uniqueness of the matrices $\{Z_a\}_a$, we have 
	$$ Z_a = Q^{-1} U_a Q = Q'^{-1} U_a' Q' =
	Q'^{-1} A^{-1} U_a A Q'  \quad , \quad a \in [t]$$
	or, equivalently
	$U_a (QQ'^{-1} A^{-1}) =
	 ({Q}{Q'}^{-1}{A}^{-1}){U}_a$ for $a \in [t]$.
	Thus, 
	${D}:= {Q}{Q}'^{-1} A^{-1}$ 
	commutes with the  
	matrices $\{U_a\}_a$ and so is diagonal itself, 
	as one of $\{U_a\}_a$ has pairwise distinct entries. 
	This gives ${Q} = {D}{A}{Q}'$ and
	proves the statement.
\end{proof}

\subsubsection{Solving Problem $\CC'$}
We consider matrices $V,E,\{Y_a\}_a$ as in Problem $\CC'$.
We want to compute matrices $\{X_a\}_a$ such that 
the matrices $Z_a=Y_a+EX_a$ commute for all $a \in [t]$,
that is, the Jacobi bracket $[Z_a,Z_b]=Z_aZ_b-Z_bZ_a$
is the zero matrix for all $a<b$.
Using $Z_a=Y_a+EX_a$, this is equivalent to 
\begin{align}\label{eq:comm}
{0}= {Y}_a {Y}_b - {Y}_b {Y}_a
+ {E} \cdot {S}_{ab} + {Y}_a {E} {X}_b
- {Y}_b {E} {X}_a
\end{align}
where ${S}_{ab} :=
{X}_a {Y}_b + {X}_a {E} {X}_b
- {X}_b {Y}_a- {X}_b {E} {X}_a$.
Left multiplication by $V$ and $VE=0$ imply
$
VY_aY_b-VY_bY_a+VY_aEX_b-VY_bEX_a=0
$, 
which is equivalent to 
\begin{equation}\label{eq:sysC'}
{\Delta}_{ab} =
{V}{Y}_b {E} {X}_a-{V}{Y}_a {E} {X}_b
\quad , \quad 1 \leq a < b \leq t\ ,
\end{equation}
where  
${\Delta}_{ab} :=
{V}{Y}_a {Y}_b - {V}{Y}_b {Y}_a$
is completely explicit in terms of the input matrices.
Eq.~\eqref{eq:sysC'} describes 
a system of linear equations over $\QQ$ 
in the variables given by the entries of
 ${X}_a$ and ${X}_b$.
Since ${\Delta}_{ab}$ has size $p \times n$,
this gives a system of 
$n p$ linear equations in the
$2 (n-p) n$ variables given
by the entries of ${X}_a$ and ${X}_b$.
Writing \eqref{eq:sysC'} for every $(a,b) \in [t]^2$ with $a < b$
we obtain a system of
$t(t-1)/2  n  p$
linear equations and
$t (n-p) n$ variables 
given by the entries of the matrices
$\{{X}_a : a \in [t]\}$.\\

From what precedes and Prop.~\ref{prop:Cprime_solves_C}, 
we deduce the following result.

\begin{prop}\label{prop:param_C}
	A unique solution to Problem $\mathbb{C}$ is implied by 
	the existence of a unique solution to the explicit system of
	linear equations given in 
	\eqref{eq:sysC'}, which is a system
	of $\frac{1}{2}t(t-1)np$ linear equations in $t(n-p)n$ variables.
	There are at least as many equations as variables 
	as soon as 
	\begin{equation}\label{eq:CondC}
	\frac{p}{n} \geq \frac{2}{t+1} \ .
	\end{equation} 
\end{prop}

Since there is no obvious linear 
dependence in the equations of the system,
we heuristically expect, in the generic case, to find a unique 
solution $\{{X}_a:a \in [t]\}$ 
under Condition \eqref{eq:CondC}.
This solves Problem $\mathbb{C}'$, and therefore 
Problem $\mathbb{C}$.

\subsection{Algorithm}
\label{ss:algo_C}
We refer to this algorithm as Algorithm $\cA_\CC$ in the sequel.

\begin{center}
	\begin{minipage}{0.95\textwidth}
		\textit{Input}:
		A valid input for Problem $\CC$\\
		\textit{Output}: 
		"Success" or "Fail"; in case of "Success", 
		also output a solution.
		"Success" means uniqueness of the solution;
		"Fail" means that no solution was found.
		
		\begin{enumerate}
			\item  Compute a basis matrix $E$ of $\ker{({W}_0)}$.
			
			\item  Define 
			${W}_0 ^+ = {W}_0 ^T ({W}_0 {W}_0^T)^{-1}$
			and for $(a,b) \in [t]^2$ with $a<b$, 
			compute the matrices
			${\Delta}_{ab}= {W}_a {W}_0^+ {W}_b -
			{W}_b {W}_0^+ {W}_a$.
			
			\item  Solve the system of linear equations 
			described in Eq.~\eqref{eq:sysC'}.
			\begin{enumerate}
				\item [(3.1)] If the solution is not unique, then output "Fail" and break.
				\item [(3.2)] Otherwise, denote by 
				$\{{X}_a : a \in [t]\}$
				the unique solution. 
			\end{enumerate}
			
			\item Perform simultaneous diagonalization of
			${Z}_a= {W}_0^+ {W}_a + {E}{X}_a$ for $a \in [t]$.
			
			\item Output "Success" with the tuples of 
			eigenvalues of the matrices 
			$\{{Z}_a\}_a$.
		\end{enumerate}
	\end{minipage}
\end{center}

\subsection{Optimization of the parameters}
\label{ss:optimize_C}

We find minimal possible (with respect to $n$) 
values for $t$ and $p$.
In our applications in Sec.~\ref{s:applications} 
we are led to minimize $p+t$ as a function of $n$.
Following Prop.~\ref{prop:param_C}, we set 
$F_n(t) = p_n(t)+t = \frac{2n}{t+1}+t$ 
with $t \in \RR_{>0}$ and $n \in \ZZ_{\geq 2}$.
It is easy to see that $F_n$ has a minimum at 
$t_0=\sqrt{2n}-1$
which gives $p=p_n(t_0)=\sqrt{2n}$.
This shows that minimal values for $p$ and $t$ are
$\cO(\sqrt{n})$. This is confirmed practically in Sec.~\ref{s:practical}.

\section{An Algorithm for Problem $\DD$}
\label{s:algo_D}

We now present an algorithm to solve Problem $\DD$ of Def.~\ref{def:prABCD}.

\subsection{Description}
\label{ss:algo_D_description}

Consider integers $n,t \geq 2$ and $2 \leq p < n$
and an instance of Problem $\DD$. 
The main idea of our algorithm is a reduction to Problem $\CC$ 
which can be solved using Algorithm $\cA_{\CC}$.
More precisely, we exhibit matrices (that are augmentations of $\{W_a\}_a$)
$W_0'=PQ'$ and $W_a' = PU_aQ'$ for $a \in [t]$,
for the same diagonal matrices $\{U_a\}_a$ and 
for some $n \times n$ invertible matrix $Q'$.

\subsubsection{Reducing Problem $\DD$ to Problem $\CC$}
For $1 \leq a, b \leq t$, we define the matrices
\begin{equation}\label{eq:Delta_abPrD}
{\Delta}_{ab} = {W}_a {W}_0^{-1} {W}_b -
{W}_b {W}_0^{-1} {W}_a \ . 
\end{equation}
Note that ${\Delta}_{ab} = - {\Delta}_{ba}$.
We have the following lemma.

\begin{lem}\label{lem:existVaGa}
	Let 
	${W}_0 = {P}{Q}$ 
	and 
	${W}_a = {P}{U}_a{Q}$ for $a \in [t]$ 
	as in Problem $\mathbb{D}$. 
	Let ${B} = {Q}{W}_0^{-1} {P}-{1}_n
	\in \mathbb{Q}^{n\times n}$
	and let $r$ denote its rank. Then:
	\begin{enumerate}
		\item [(i)] $r = n-p$
		\item [(ii)] there exist matrices
		${V}_a \in \mathbb{Q}^{p\times r}$ and ${G}_a \in
		\QQ^{r \times p}$ for $a \in [t]$ such that for all 
		$1 \leq a< b \leq t$, 
		one has
		$
		{\Delta}_{ab} = {V}_a {G}_b - {V}_b {G}_a . 
		$
	\end{enumerate}
\end{lem}
\begin{proof}
	\textit{(i)} 
	Let ${C}={Q} {W}_0^{-1} {P}$. Then ${C}{Q} = {Q}$
	and the column-image of 
	${Q}$ is contained in the eigenspace, say $\cE$, 
	of ${C}$ for eigenvalue $1$. 
	So, $\cE$ has dimension at least $p$. 
	However, the rank of ${C}$ is 
	bounded above by the rank of ${Q}$, i.e. by $p$. 
	Finally, $\cE$ has dimension exactly $p$ and the 
	rank $r$ of ${B} = {C} - {1}_n$ equals $n-p$.  
	
	\textit{(ii)} 
	For every $1\leq a,b \leq t$, we can write
	\begin{eqnarray}\label{eq:symeqndeltaab}
	{\Delta}_{ab} &=&
	{P} {U}_a ( {Q} {W}_0^{-1} {P} - {1}_n) {U}_b {Q}
	-  {P} {U}_b ( {Q} {W}_0^{-1} {P} - {1}_n) {U}_a 
	{Q} \notag \\
	&=& {P} {U}_a {B} {U}_b {Q} - {P} {U}_b {B} 
	{U}_a {Q} 
	\end{eqnarray}
	since
	${U}_a$ and ${U}_b$ commute.
	Since ${B}$ has rank $r$, 
	there exist matrices ${B}_1 \in \QQ^{n\times r}, {B}_2\in \QQ^{r\times n}$ 
	with ${B} = {B}_1{B}_2$. 
	Setting ${V}_a = {P}{U}_a {B}_1$ and 
	${G}_a = {B}_2 {U}_a {Q}$
	gives the claim.  
\end{proof}

The following properties of the matrix $B$ defined in Lem.~\ref{lem:existVaGa} are useful.

\begin{lem}\label{lem:propMatB}
	Let ${W}_0 = {P}{Q}$ and 
	${W}_a = {P}{U}_a{Q}$
	for $a \in [t]$ as in Problem $\mathbb{D}$.
	Let ${B} \in \QQ^{n \times n}$ be the matrix of 
	Lem.~\ref{lem:existVaGa}
	with respect to ${P}$ and ${Q}$ and let $r=n-p$.
	Let ${B}_1 \in \QQ^{n \times r}$ and ${B}_2 \in \QQ^{r \times n}$ be
	such that
	${B} = {B}_1 {B}_2$. Then:
	\begin{enumerate}
		\item [(i)] ${P}  {B}_1 = {0}_{p \times r}$
		\item [(ii)] 
		The matrix $Q':=[{Q} | {B}_1]$ is an $n \times n$ invertible matrix.
	\end{enumerate}
\end{lem}
\begin{proof}
	\textit{(i)} The matrix ${B}_2$ 
	defines a surjection ${B}_2: \mathbb{Q}^n \to \mathbb{Q}^r$.
	Thus for every ${x} \in \mathbb{Q}^r$, we write ${x}=
	{B}_2 {y}$ for some ${y} \in \mathbb{Q}^n$ and obtain
	${P}{B}_1 {x} =
	{P} {B}_1 ({B}_2 {y})=
	({P} {B}){y} =
	{0}$.
	
	\textit{(ii)} Since $r=n-p$, $Q'$
	has size $n \times n$. To show its invertibility, we show that
	$\mathrm{im}({Q}) \cap \mathrm{im}({B}_1) = \{{0}\}$.
	Since ${B}_2$ is surjective,
	the images of ${B}_1$ and ${B}_1 {B}_2 = {B}$
	coincide.
	Let ${Q} {x} = {B} {y} 
	\in \mathrm{im}({Q}) \cap \mathrm{im}({B}_1)$,
	with ${x} \in \QQ^p$ and ${y} \in \QQ^n$.
	This gives
	${Q}{x} 
	= ({Q} {W}_0^{-1} {P} - 
	{1}_n){y}= {Q} {W}_0^{-1} {P} {y} - 
	{y}$.
	Thus ${y} = 
	{Q} {W}_0^{-1} {P}{y} - {Q} {x}
	= {Q} {z}$
	with ${z}= {W}_0^{-1} {P} {y} - {x}$.
	Therefore,
	${Q}{x} = {B} {y} = {B} ({Q}{z}) = {0}$
	because ${B} {Q} = {0}$. 
\end{proof}

We now show that finding matrices $\{V_a\}_a$ such that there exist $\{G_a\}_a$ satisfying 
$\Delta_{ab} = V_a G_b - V_b G_a$ for every $a,b$ is 
sufficient to solve Problem $\DD$.
We view these matrices as being complementary to $\{W_a\}_a$
because they define themselves an instance of Problem $\DD$ 
with the same solution as $\{W_a\}_a$ 
(see the proof of Lem.~\ref{lem:existVaGa}).
This allows us to increase the rank of $Q$. 
We thus now formulate Problem $\DD'$.
 
\begin{defi}[Problem $\mathbb{D}'$]
	Let $n, t\geq 2$ and $2 \leq p < n$ be integers.
	For every $1 \leq a, b \leq t$, 
	let ${\Delta}_{ab} \in \QQ^{p \times p}$ be
	such that 
	${\Delta}_{ab} = {V}_a {G}_b - {V}_b {G}_a$
	for
	${V}_a \in \mathbb{Q}^{p \times (n-p)}$ of rank $n-p$
	and 
	${G}_a \in \QQ^{(n-p) \times p}$.
	The problem states as follows:
	Given the matrices ${\Delta}_{ab}$ for all 
	$1 \leq a,b \leq t$,
	compute such matrices ${V}_a$ for $a \in [t]$.
\end{defi}

The following proposition links  
Problem $\mathbb{D}$ and Problem $\mathbb{C}$.

\begin{prop}\label{prop:D'_solves_D}
	Let 
	${W}_0=PQ$ and 
	${W}_a=PU_aQ$ for $a \in [t]$ be as in 
	Problem $\DD$.
	For $1 \leq a,b \leq t$, 
	let ${\Delta}_{ab}$ be the matrices 
	defined in \eqref{eq:Delta_abPrD}.
	Moreover, assume that 
	\begin{itemize}
		\item [(i)] Problem $\mathbb{D}'$ is uniquely solvable for
		the input matrices 
		$\{{\Delta}_{ab} : 1 \leq a < b \leq t\}$
		and denote by $\{{V}_a: a \in [t]\}$ the unique
		solution.
		\item [(ii)] Problem $\mathbb{C}$ is uniquely solvable for 
		the input matrices 
		${W}_0' = [{W}_0|{0}_{p\times (n-p)}] \in \QQ^{p \times n}$ and 
		${W}_a' = [{W}_a | {V}_a] \in \QQ^{p \times n}$ for $a \in [t]$.
	\end{itemize}
	Then Problem $\mathbb{D}$ is uniquely 
	solvable on input  
	$\{{W}_a : 0 \leq a \leq t\}$ 
	and the unique solution is given by the unique
	solution to Problem
	$\mathbb{C}$ on input  
	$\{{W}_a' : 0 \leq a \leq t\}$.
\end{prop}

\begin{proof}
	By Lem.~\ref{lem:existVaGa} there exist 
	${V}_a \in \mathbb{Q}^{p \times r}$
	and ${G}_a \in \mathbb{Q}^{r \times p}$ for $a \in [t]$
	such that
	${\Delta}_{ab} = {V}_a {G}_b - {V}_b {G}_a$
	for all $1 \leq a < b\leq t$.
	Therefore the matrices $\{{\Delta}_{ab}\}_{a,b}$
	define an instance of Problem $\mathbb{D}'$. 
	By assumption \textit{(i)},
	we compute the unique 
	solution $\{{V}_a\}_a$ for this problem.
	
	Now, let
	${W}_0' = [{W}_0 | {0}_{p\times (n-p)}] \in \QQ^{p \times n}$
	and
	${W}_a' = [{W}_a | {V}_a] \in \QQ^{p \times n}$ for $a \in [t]$.
	Let $B=QW_0^{-1}P-1_n$ as in Lem.~\ref{lem:existVaGa} of rank $r=n-p$.
	Let ${B}_1 \in \QQ^{n \times r}$ and ${B}_2 \in \QQ^{r \times n}$
	be a rank factorization of ${B}$; i.e.
	${B} = {B}_1{B}_2$. 
	Letting 
	${Q}' := [{Q} | {B}_1] \in \QQ^{n \times n}$,
	we have 
	$
	{P}{Q}' = {P} [{Q}|{B}_1] = 
	[{W}_0 | {0}_{p \times r}] = {W}_0'
	$
	and, by uniqueness of $\{V_a\}_a$ (see proof of Lem.~\ref{lem:existVaGa}),
	$$
	{P}{U}_a {Q}' = {P}{U}_a [{Q}|{B}_1] = 
	[{W}_a | {V}_a] = {W}_a' 
	$$
	for $a \in [t]$,
	as ${P}{B}_1 = {0}_{n \times r}$ 
	by Lem.~\ref{lem:propMatB} \textit{(i)}.
	The matrix ${Q}'$ is invertible by Lem.~\ref{lem:propMatB} \textit{(ii)}. 
	Therefore, ${W}_0'$ and $\{{W}_a'\}_a$ 
	define a valid input for Problem $\mathbb{C}$.
	By assumption \textit{(ii)}, this problem is uniquely solvable
	and the solution must be the tuples of 
	diagonal entries of the matrices $\{{U}_a\}_a$.
	This is also a solution to Problem $\mathbb{D}$
	since the matrices $\{{U}_a\}_a$ 
	are the same for the input matrices 
	$\{{W}_a\}_a$ for Problem $\mathbb{D}$ and 
	$\{{W}_a'\}_{a}$ for Problem $\mathbb{C}$.
\end{proof}

\subsubsection{Solving Problem $\DD'$}

In view of Prop.~\ref{prop:D'_solves_D}, it remains to compute matrices 
$\{V_a\}_a$ from $\{\Delta_{ab}\}_{a,b}$.
We achieve this by standard linear algebra, and combining with 
Algorithm $\cA_\CC$ describes a full algorithm 
for Problem $\DD$.\\

From now on we assume $t \geq 3$. Let
$\Delta_{ab} = V_aG_b-V_bG_a$
for $1 \leq a,b \leq t$
as in
Problem $\DD'$. 
Let $r=n-p$ and $r_{ab}$ be the rank of $\Delta_{ab}$;
clearly, $r_{ab} \leq \min(2r,p)$.
We further assume $p > 2n/3$ (equivalently $2r<p$), which is a necessary condition
as otherwise the matrices 
$\Delta_{ab}$ likely have full rank and thus cannot reveal any information.
We define $\cK_{ab}:=\mathrm{im}(\Delta_{ab}) = \cK_{ba} \subseteq \QQ^p$
and 
$$\cK_a = \bigcap_{b \in [t],b\neq a} \cK_{ab} \ , \ a \in [t].$$
Let $\cV_a$ be the image of the matrix $V_a$ for $a \in [t]$. 
We first argue that, heuristically, 
$\mathcal{V}_a \subseteq \mathcal{K}_{ab}$ for every $b\neq a$.
Let $v \in \mathcal{V}_a$.
If there exists $x \in \QQ^p$ such that $v = V_aG_bx$ and $V_bG_a x= 0$
then $v=\Delta_{ab}x$, i.e.~$v \in \mathcal{K}_{ab}$.
Such an element $x$ must therefore lie in
$(x_0+\ker(V_aG_b))\cap \ker(V_bG_a)$,
where $x_0 \in \QQ^p$ is any vector such that $v=V_aG_b x_0$.
It is easy to see that this intersection is non-empty if 
$\ker(V_aG_b)+\ker(V_bG_a) = \QQ^p$.
Heuristically, as $\{V_a\}_a$ have rank $r$, 
$\ker(V_aG_b)+\ker(V_bG_a)$ has dimension at least $2(p-r)$;
accordingly we can heuristically 
expect that $\ker(V_aG_b)+\ker(V_bG_a) = \QQ^p$
as soon as $2(p-r) > p$, i.e. $p > 2n/3$.

We now justify that, heuristically
under a suitable parameter selection, 
$\cK_a = \cV_a$ for every $a\in [t]$.
For fixed $a \in [t]$, we compute $\cK_a$ modulo 
$\cV_a$ and consider 
$\overline{\cK_{ab}}:=
\cK_{ab}/\cV_a \subseteq\QQ^{p-r}$ for $b \neq a$.
Then $\cK_a = \cV_a$ if and only if 
$
\overline{\cK_a}:=\bigcap_{b\neq a} \overline{\cK_{ab}} = \{0\}
$.
Since $\cV_a$ has dimension $r$, $\overline{\cK_{ab}}$ has dimension 
$r_{ab}-r$.
For every $b \neq a$, we view $\overline{\cK_{ab}}$ as the kernel 
of $\QQ^{p-r} \to \QQ^{p-r}/\overline{\cK_{ab}}$, 
represented by a matrix $A_{ab} \in \QQ^{(p-r_{ab})\times (p-r)}$.
Therefore $\overline{\cK_a}$ is represented by an augmented
matrix 
$A_a = [A_{a1}|\ldots|A_{a,a-1}|A_{a,a+1}|\ldots|A_{at}]$
describing the kernel of 
$\QQ^{p-r} \to \bigoplus_{b\neq a} \QQ^{p-r}/\overline{\cK_{ab}}$.
The matrix $A_a$ has $\sum_{b\in [t], b \neq a} (p-r_{ab})$ rows and $p-r$ 
columns.
Now, $\cK_a=\cV_a$ if and only if $A_a$ has full rank;
and heuristically, we expect this to be the case as soon as 
$\sum_{b\in [t], b \neq a} (p-r_{ab}) \geq p-r$.

\begin{rem}\label{rem:prD}
		(i) 
		In fact, we expect that
		$r_{ab} = 2r$ for every $a,b$.
		Then, from what precedes, 
		we expect, heuristically that $\cK_a = \cV_a$ for every $a$, 
		if $(t-1)(p-2r) \geq p-r$,
		i.e.
		\begin{equation}\label{eq:heuristic_bound_pt_PbD}
		\frac{p}{n} \geq \frac{2t-3}{3t-5} \ \ \textrm{or, equivalently,} \ \
		t \geq \frac{2p-n}{3p-2n}+1  \ .
		\end{equation}
		
		(ii) We assumed $t \geq 3$ so that the 
		intersections $\{\cK_a\}_a$ 
		are well-defined. If $t=2$, $\cK_1$ coincides with the image 
		of ${\Delta}_{12}$, 
		which will not reveal  
		${V}_1$ and ${V}_2$. 

\end{rem}

We compute bases of $\{\cK_a\}_a$ 
by standard linear algebra. 
For the rest of this section, assume $\cK_a = \cV_a$ 
for every $a$,
and let ${C}_a$ be a basis matrix for $\cK_a$.
Thus, there exists ${M}_a \in \GL_{r}(\QQ)$
such that 
${V}_a = {C}_a  {M}_a$.
This gives for $a < b$:
\begin{equation}\label{eq:symsysNab}
{\Delta}_{ab} = {V}_a {G}_b - {V}_b {G}_a  
= {C}_a ({M}_a {G}_b) - {C}_b ({M}_b {G}_a) 
= {C}_a  {N}_{ab} - {C}_b  {N}_{ba}
\end{equation}
with ${N}_{ab} = {M}_a {G}_b$.
In \eqref{eq:symsysNab}, 
${\Delta}_{ab}$ and ${C}_a, {C}_b$ are known,
which allows to compute 
${N}^{(ab)} = [{N}_{ab}  | {N}_{ba}]^T$
as a solution to 
$
{\Delta}_{ab} = [{C}_a  |  -{C}_b] \cdot {N}^{(ab)}
$.
Once $\{N_{ab}\}_{a,b}$ are computed, 
we obtain a system of linear equations over $\QQ$, 
given by
\begin{equation}\label{eq:sys2}
{M}_a^{-1}\cdot {N}_{ab} = {G}_b     
\quad , \quad 1 \leq a < b \leq t \ .
\end{equation}

It has
$\frac{1}{2}t(t-1) r p$ equations
(there are $\frac{1}{2}t(t-1)$ 
choices for pairs $(a,b)$ and for each
pair the matrix equality gives $r p$ equations)
and
$t r^2 + trp = trn$
variables, given by the $tr^2$ entries of the 
matrices $\{{M}_a^{-1} : a \in [t]\}$
and the $trp$ entries of the matrices 
$\{{G}_b : b \in [t]\}$.
Heuristically, if $trn \leq \frac{1}{2}t(t-1) rp$,
i.e.
$
2n \leq (t-1)p 
$,
the system is expected to have a unique solution.
This bound is automatically 
satisfied if \eqref{eq:heuristic_bound_pt_PbD} holds.
This reveals $\{{M}_a : a \in [t]\}$
and thus $\{{V}_a : a \in [t]\}$
by computing $V_a=C_aM_a$.

\begin{prop}
	Assume that $\cK_a = \cV_a$ for every $a \in [t]$ (see Rem.~\ref{rem:prD} (i)).
	Then, a unique solution to Problem $\mathbb{D}'$ is implied by 
	the existence of a unique solution to the explicit system of linear equations given in  
	\eqref{eq:sys2}, which is a system
	of $\frac{1}{2}t(t-1)(n-p)p$ 
	linear equations in $t(n-p)n$ variables.
	There are at least as many equations as variables 
	as soon as $p(t-1) \geq 2n$.
\end{prop}

\subsection{Algorithm}
\label{ss:algo_D}
We refer to this Algorithm as Algorithm $\cA_\DD$ in the sequel.
\begin{center}
	\begin{minipage}{0.95\textwidth}
		\textit{Input}:
		A valid input for Problem $\DD$.\\
		\textit{Output}: 
		"Success" or "Fail", and in case of "Success", additionally output a 
		solution.
		"Success" means that the computed solution is unique;
		"Fail" means that a solution was not found.
		
		\begin{enumerate}
			\item Compute 
			${\Delta}_{ab}
			= {W}_a {W}_0^{-1} {W}_b - {W}_b
			{W}_0^{-1} {W}_a$
			for $1 \leq a\neq b \leq t$
			
			\item For $a \in [t]$, compute a basis matrix ${C}_a$ of
			$\cK_a:=\bigcap_{b \in [t], b \neq a} \mathrm{im}({\Delta}_{ab})$
			
			\item 
			\begin{enumerate}
				\item [(3.1)] 
				If $\dim(\cK_a) \neq n-p$ for all $a \in [t]$, then 
				output "Fail" and break
				\item [(3.2)] 
				Otherwise, for every $a < b$ compute  
				${N}_{ab}$
				as solutions to 
				${\Delta}_{ab} = [{C}_a | - {C}_b] \cdot 
				[{N}_{ab} | {N}_{ba}]^T$
			\end{enumerate}
			
			\item Solve for $\{M_a\}_a$ the system of linear equations
			${M}_a^{-1} {N}_{ab} = {G}_b$
			for $(a,b) \in [t]^2, a < b$.
			\begin{enumerate}
				\item [(4.1)] If a unique solution is not found, 
				then output "Fail" and break
			\end{enumerate}
			\item Compute the matrices $\{{V}_a : a \in [t]\}$
			as ${V}_a = {C}_a \cdot {M}_a$
			
			\item Run Algorithm $\mathcal{A}_{\mathbb{C}}$
			on the matrices
			${W}_0 ' = [{W}_0 | {0}]$
			and
			${W}_a ' = [{W}_a | {V}_a]$
			for $a \in [t]$
			and return its output.
		\end{enumerate}
	\end{minipage}
\end{center}

\begin{rem}
	Problem $\DD$ of Def.~\ref{def:prABCD} is symmetric in the sense that
	$P$ and $Q$ have the same rank.
	An asymmetric variant consists in having $P$ and $Q$ of ranks $p\neq q$.
	Our algorithm adapts to that case: if $p<q$, then "cutting" the last $q-p$
	columns of $\{W_a\}_a$ means "cutting" the last $q-p$
	columns of $Q$, which reduces to the symmetric case.
	This approach is however not very genuine, as it "cuts" information
	instead of possibly exploiting it. We leave it open to find a better 
	algorithm.
\end{rem}

\subsection{Optimization of the parameters}
\label{ss:optimize_D}

We find minimal possible values for $t$ and $p$ with respect to a given $n$.
In Sec.~\ref{ss:CLT13} we will see that it is of interest
to minimize $2p+t$ 
in order to minimize the number of public encodings in  \cite{CLT13}.
According to \eqref{eq:heuristic_bound_pt_PbD}, 
the main (heuristic) condition to be ensured is 
$p \geq \frac{2t-3}{3t-5}n$.
We set $F_n(t)=2p_n(t)+t=\frac{2t-3}{3t-5}n+t$ 
for $t \in \RR_{>0}\backslash \{5/3\}$ and $n\geq 2$.
Then $F_n$ has a minimum at $t_0=\frac{1}{3}(\sqrt{2n}+5)$, with 
$p_n(t_0)=\frac{2}{3}n+\frac{1}{3\sqrt{2}}\sqrt{n}$
and $F_n(t_0)=\frac{4}{3}n+\frac{2}{3}\sqrt{2n}+\frac{5}{3}$.
In conclusion, we expect Algorithm $\cA_\DD$ to succeed 
for $p=\lceil p_n(t_0) \rceil$ and 
$t = \lceil t_0 \rceil$. 

\section{Applications}
\label{s:applications}

We describe two applications for our algorithms and obtain
significant improvements on previous works. 

\subsection{Improved algorithm for the CRT-ACD Problem}
\label{ss:crtacd}

We consider the following "multi-prime" 
version of the Approximate 
Common Divisor Problem \cite{corper} 
based on Chinese Remaindering:

\begin{defi}[{\sf{CRT-ACD}} Problem]\label{def:crtacd}
	Let $n, \eta, \rho \in \ZZ_{\geq 1}$.
	Let $p_1,\ldots,p_n$
	be distinct $\eta$-bit prime numbers
	and $M=\prod_{i=1}^{n} p_i$.
	Consider a non-empty finite set
	$\mathcal{S}$ of integers in $\ZZ\cap[0,M)$
	such that for every $x \in \mathcal{S}$:
	$$x \equiv x_{i} \pmod{p_i} \ , \quad 1 \leq i \leq n $$
	for uniformly distributed integers $x_{i} \in \ZZ$ satisfying
	$|x_{i}|\leq 2^\rho$.
	
	The {\sf{CRT-ACD}} problem states as follows:
	given the set $\mathcal{S}$,
	the integers $\eta,\rho$ and $M$, factor $M$ completely
	(i.e. find the prime numbers $p_1,\ldots,p_n$).
\end{defi}

Clearly, the larger the set $\cS$, the more information
one can exploit to factor $M$. 
Our interest is therefore to minimize the 
cardinality of the set $\cS$ with respect to $n$.

\subsubsection{The algorithm of \cite{corper}}
\label{ss:coron_pereira}
Coron and Pereira propose an algorithm for the case $\#\cS=n+1$.
They proceed in two steps called the "orthogonal lattice attack" following \cite{ns99}
and the "algebraic attack" following \cite{cheon15}.
We briefly review their algorithm; for a complete description we refer to 
\cite[Se. 4.3]{corper}.\\

Let $\cS=\{x_1,\ldots,x_n,y\}$ and $x = (x_1,\ldots,x_n) \in \cS^n$.
Then, the vector 
$b = (x,y\cdot x) \in \ZZ^{2n}$ is public, and
by the Chinese Remainder Theorem, letting 
$x \equiv x^{(i)} \pmod{p_i}$
and $y \equiv y^{(i)} \pmod{p_i}$ 
for all $i \in [n]$, 
one has 
$b \equiv \sum_{i=1}^n c_i (x^{(i)},y^{(i)}x^{(i)}) =: \sum_{i=1}^n c_i b^{(i)} \pmod{M}$,
for some integers $c_1,\ldots,c_n$.
If the vectors $\{x^{(i)}\}_i$ are $\RR$-linearly independent, then so are 
$\{b^{(i)}\}_i$ and generate a $2n$-dimensional lattice $\calL$ of rank $n$.
Importantly, by Def.~\ref{def:crtacd}, 
the vectors $\{b^{(i)}\}_i$ are reasonably short vectors 
(of $\ell_2$-norm approximately $2^{2\rho}$; and $\rho$ is considered 
much smaller than $\eta$).

The "orthogonal lattice attack" is an algorithm, which on input $b$, 
outputs a basis of the completion 
$\overline{\calL}=
\calL \otimes_{\ZZ} \QQ$ of $\calL$, performing lattice reduction
on the lattice $\langle b \rangle^{\perp_M}$ 
of vectors $v \in \ZZ^m$ such that $\langle v,b\rangle \equiv 0\pmod{M}$.
The parameters are chosen accordingly,
and one essentially requires $2\rho < \eta$.

Upon finding a basis $\{b'^{(i)}\}_i$ of $\overline{\calL}$, 
Coron and Pereira proceed with the "algebraic attack". 
The bases $\{b'^{(i)}\}_i$ of $\overline{\calL}$ and $\{b^{(i)}\}_i$ of $\calL$
are related via an unknown invertible base change matrix $Q \in \QQ^{n \times n}$. 
Letting $P = [x^{(1)} | \ldots | x^{(n)}] \in \ZZ^{n \times n}$ with columns 
$\{x^{(i)}\}_i$, one obtains matrix relations
\begin{equation}\label{eq:matrix_crtacd}
W_0 = P\cdot Q \ , \ W_1 = P \cdot U_1 \cdot Q
\end{equation}
where $U_1$ is $n \times n$ diagonal with entries $\{y^{(i)}\}_i$.
The matrix $W_0$ is invertible (over $\QQ$) and one computes 
the eigenvalues $\{y^{(i)}\}_i$ of
$W_1 W_0^{-1} = P  U_1  P^{-1}$.
Using $y \equiv y^{(i)} \pmod{p_i}$, one factors $M$ by 
computing greatest common divisors.

\subsubsection{A naive improvement}
There is a naive generalization of \cite{corper} using only $\cO(\sqrt{n})$ 
public instances in $\cS$. 
However, we argue that this approach gives a worse range of parameters
when combined with \cite{corper}.

For integers $p \geq 2$ and $t \geq 1$ of size $\cO(\sqrt{n})$, let 
$x = (y_1\cdot z, \ldots, y_t\cdot z) \in \ZZ^{tp}$ 
of dimension $\cO(n)$
for $y_1,\ldots,y_t \in \cS$ and $z \in \cS^p$.
This variant reduces $\#\cS$ considerably, as 
$\#\cS = p+t = \cO(\sqrt{n})$.
However, \cite{corper} requires to construct 
the vector $b=(x,y\cdot x)$ for $y \in \cS$.
This gives rise to residue vectors $\{b^{(i)}\}_i$ of approximate $\ell_2$-norm 
$2^{3\rho}$ instead of $2^{2\rho}$ as in \cite{corper}.
Therefore the stronger condition $3\rho < \eta$ will be required for 
the orthogonal lattice attack to succeed.
In our improvement, we would like to lower $\#\cS$
while continuing to use $2\rho < \eta$, as in \cite{corper}.

\subsubsection{Our improved algorithm}
\label{ss:improved_crtacd}

We recognize that \eqref{eq:matrix_crtacd} defines an instance of 
Problem $\mathbb{A}$ of 
Def.~\ref{def:prABCD} with $t=1$ because $P$ and $Q$ have rank $n$.
Our improvement lies in generalizing the vector $b$ as to obtain 
an instance of Problem $\CC$. 

We consider $\#\cS < n+1$ and write for convenience 
$\cS=\{x_1,\ldots,x_p,y_1,\ldots,y_t\}$ with integers $2 \leq p < n$ and $2 \leq t < n$ 
satisfying $2n \leq (t+1)p$.
We let $x = (x_1,\ldots,x_p) \in \cS^p$ and 
$b = 
(x, y_1 \cdot x, \ldots, y_t \cdot x) \in \ZZ^{(t+1)p}$.
As previously, 
let $\{b^{(i)}\}_i$ denote the short
residue vectors modulo the primes $\{p_i\}_i$
and $x \equiv x^{(i)} \pmod{p_i}$, $y_a \equiv y_a^{(i)} \pmod{p_i}$
for $a \in [t]$ and $i \in [n]$.
By the Chinese Remainder Theorem, 
we observe that $b$ lies in the 
lattice $\calL = \bigoplus_{i = 1}^n \ZZ b^{(i)}$
modulo $M$. Namely, there are integers $c_1,\ldots,c_n$ such that
\begin{equation*}
b \equiv \sum_{i=1}^n c_i
\begin{bmatrix}
x_i \\  y_{1}^{(i)} \cdot x^{(i)} \\ \vdots \\ y_{t}^{(i)}\cdot x^{(i)}
\end{bmatrix}
= :
\sum_{i=1}^n c_i b^{(i)}
\pmod{M} \
\end{equation*}
As in \cite{corper}, the orthogonal lattice algorithm reveals a basis
$\{b'^{(i)}\}_i$ of $\overline{\calL}$
and 
the $\ell_2$-norm of $\{b^{(i)}\}_i$ is still approximately $2\rho$.

Contrary to \eqref{eq:matrix_crtacd}, we now derive matrix equations
\begin{equation}\label{eq:improved_matrix_crtacd}
W_0 = P\cdot Q \ , \ W_a = P \cdot U_a \cdot Q \ , \ a \in [t]
\end{equation}
where $P \in \ZZ^{p \times n}$ has columns $\{x^{(i)}\}_i$ 
and $\{U_a\}_a$ are $n \times n$
diagonal with entries $\{y_a^{(i)}\}_{a,i}$.
The matrix $Q$ is a base change matrix from $\{b'^{(i)}\}_i$ to $\{b^{(i)}\}_i$.
If $W_0$ has rank $p$, Eq.~\eqref{eq:improved_matrix_crtacd}
now defines a valid input for Problem $\CC$ of Def.~\ref{def:prABCD}
and Algorithm $\cA_{\CC}$ from Sec.~\ref{s:algo_C} 
reveals the diagonal entries $\{y_a^{(i)}\}_{a,i}$ of the matrices 
$\{U_a\}_a$.
One can then factor $M$ by computing $\gcd(y_a-y_{a}^{(i)},M)$.
 
From Sec.~\ref{ss:optimize_C} we see that 
$\#\cS = p+t$ is 
minimized for $p=\lceil \sqrt{2n}\rceil$ and 
$t+1=\lceil \sqrt{2n} \rceil$.
Thus, $\#\cS=2 \lceil \sqrt{2n} \rceil = \cO(\sqrt{n})$.
In summary, letting $n$ be the number of secret primes in the 
public modulus $M$, we can factor $M$ given only 
$\cO(\sqrt{n})$ input samples, 
whereas \cite{corper} uses $\cO(n)$.

\ignore{
\paragraph{\itshape{Parameter Selection and 
	Asymptotic Complexity}}
The parameters and asymptotic time complexity
of the algorithm are obtained by analyzing lattice reduction on the 
lattice $\langle b \rangle^{\perp_M}$. 
Herefore, standard algorithms such as LLL \cite{lll82} or BKZ
\cite{terminatingBKZ} are used.
Using BKZ with block-size 
$\beta=\omega(1/\log(\iota))$, a Hermite factor $\iota$ is heuristically 
achieved within time $2^{\Omega(1/\log(\iota))}$, \cite{hujia16}.
In our case, with $p=t=\cO(\sqrt{n})$, 
the vector $b$ (and thus the lattice $\calL$) 
has dimension $p(t+1)=\cO(n)$ as in \cite{corper}.
Similarly, the vectors $\{b^{(i)}\}$ have $\ell_2$-norm
$\Vert b^{(i)} \Vert = \Vert x^{(i)} \Vert \cdot (1+\sum_{a \in [t]} (y_a^{(i)})^2)^{1/2} \leq \sqrt{p} \cdot 2^{\rho} (1+t 2^{2\rho})^{1/2}$,
which is again approximately $2^{2\rho}$.
Based on \cite{hujia16}, we therefore obtain as in \cite{corper}, 
a heuristic time complexity of $2^{\Omega(\gamma/\eta^2)}$.
}

\begin{rem}
	We remark that we do not impact the security of the key-exchange 
	from \cite{corper},
	as it uses certain encodings of matrices. However,  
	the product of matrices does not commute, so our techniques do not apply to that case.
\end{rem}

\subsection{Improved Cryptanalysis of CLT13 Multilinear Maps}
\label{ss:clt13}
We consider the CLT13 Multilinear 
Map Scheme by Coron {\sl et al.}, \cite{CLT13}.
Cheon {\sl et al.} \cite{cheon15}
described a polynomial-time attack against the Diffie-Hellman key exchange based on CLT13
when enough encodings of zero are public.
Such encodings are for instance public in the rerandomization procedure.
It is of interest to investigate this cryptanalysis 
when only a limited number of such encodings is available.
Namely, not every CLT13-based construction necessarily 
reveals enough such encodings and the attack of 
Cheon {\sl et al.} is prevented.

\subsubsection{CLT13 Multilinear Maps}
\label{ss:CLT13}

The CLT13 Multilinear Map is a construction over the integers based on the 
notion of graded encoding scheme \cite{ggh13}.
Its hardness relies on Chinese Remainder-representations 
and factorization.
We fix an integer $n \geq 2$, thought of as a dimension for CLT13. 
The message space is $\bigoplus_{i = 1}^n \ZZ/g_i \ZZ$
for some small secret primes $\{g_i\}_i$.
The encoding space has a graded structure 
and supports homomorphic addition and multiplication.
It is defined over $\bigoplus_{i=1}^n \ZZ/p_i\ZZ$
for large secret primes $\{p_i\}_i$
with public product $x_0 = \prod_i p_i$.
More precisely, an encoding of a 
message $m = (m_i)_i \in \bigoplus_{i = 1}^n \ZZ/g_i \ZZ$ at level $k \in [\kappa]$ 
(where $\kappa$ denotes the multilinearity degree)
is an integer $c$ such that
$c \equiv (r_i g_i + m_i)\cdot z^{-k} \pmod{p_i}$ for all $i \in [n]$
where $z \in (\ZZ/x_0\ZZ)^{\times}$ and $r_i$ is a random "small" noise. 
By the Chinese Remainder Theorem $c$ is computed modulo 
$x_0$.
For encodings $c$ at the last level $\kappa$, a public
zero-testing procedure allows to test if $c$ encodes zero. 
This procedure works by computing 
$\omega(c) := p_{zt} \cdot c$
for a public parameter $p_{zt} \in \ZZ/x_0\ZZ$.
Then $c$ encodes the zero message if $\omega$
is ''small'' compared to $x_0$.
In \cite{CLT13}, one actually defines a vector of $n$ zero-test parameters 
$\{p_{zt, i} : i \in [n]\}$ to define a proper zero-testing.
For the precise parameter setting, 
we refer to \cite[Sec. 3.1]{CLT13}.

\subsubsection{Cryptanalysis}
\label{ss:CryptCLT13}

The algorithm from \cite{cheon15} 
reveals all secret parameters given sufficiently
many encodings of zero. 
We briefly recall the attack here, and
for simplicity of exposition, assume $\kappa=3$.
Consider sets $\cA=\{\alpha_j : j \in [n]\}$, 
$\cB = \{\beta_1,\beta_2\}$
and $\cC=\{\gamma_k : k \in [n]\}$
of encodings at level $1$ and where all encodings in $\cA$ encode zero. 
Therefore, there are $\#\cA = n$ public encodings of zero 
and $\#(\cA\cup\cB\cup\cC) = 2n+2$ encodings in total.
In the previous notation, we write 
$\alpha_{j} \equiv \alpha_{ji}/z \pmod{p_i}$,
$\beta_a \equiv \beta_{ai}/z \pmod{p_i}$
and $\gamma_k \equiv \gamma_{ki}/z \pmod{p_i}$
for all $i,j,k \in [n]$ and $a \in [2]$.
Because the products 
$\alpha_j \beta_a \gamma_k$
encode zero at level $3$,
correct zero-testing ensures that the zero-test equations
$\omega_{jk}^{(a)} = p_{zt}(\alpha_j\beta_a\gamma_k)$,
given by
\begin{eqnarray*}
	\omega^{(a)}_{jk} =
	\sum_{i=1}^n p_{zt,i} \alpha_{ji} \beta_{ai} \gamma_{ki} =
	\begin{bmatrix}
		\alpha_{j1} \cdots \alpha_{jn}
	\end{bmatrix}
	\begin{bmatrix}
		\beta_{a1} p_{zt,1} &  &  \\
		& \ddots & \\
		& & \beta_{an}p_{zt,n}
	\end{bmatrix}
	\begin{bmatrix}
		\gamma_{k1} \\ \vdots \\ \gamma_{kn} 
	\end{bmatrix} 
\end{eqnarray*}
for certain explicit integers $p_{zt,i}$ for $i \in [n]$
defining the zero-test parameter, 
hold over $\ZZ$ instead of $\ZZ/x_0\ZZ$. 
Writing these relations out for all indices $(j,k) \in [n]^2$,  
the $n \times n$ matrices $W_a:=(\omega_{jk}^{(a)})_{j,k \in [n]}$ for 
$a=1,2$ satisfy
\begin{equation}\label{eq:cheon}
W_a = P \cdot U_a \cdot Q
\end{equation}
for secret matrices 
$P,Q$ of full rank $n$ 
(corresponding to encodings of $\cA$ and $\cC$, respectively)
and diagonal matrices $\{U_a\}_a$ 
containing the elements $\{\beta_{ai}: i \in [n]\}$.
If
at least one of ${W}_1, {W}_2$
is invertible over $\QQ$ (say ${W}_2$),
the attacker computes the eigenvalues 
$\{\beta_{1i}/\beta_{2i}:i \in [n]\}$ of
$W_1W_2^{-1}$.  
These ratios are
enough to factor $x_0$. 
Indeed, letting $\beta_{1i}/\beta_{2i}=x_i/y_i$
for coprime integers $x_i,y_i$ and using 
$\beta_a \equiv \beta_{ai}/z
\pmod{p_i}$, we deduce
$ x_i \beta_{2} - y_i \beta_{1}
 \equiv (x_i \beta_{2i} - y_i \beta_{1i})/z \equiv 0
\pmod{p_i}$ for $i \in [n]$
and therefore 
$\gcd(x _i \beta_{2} - y_i \beta_{1},x_0)=p_i$
with high probability.

In summary, the Cheon {\sl et al.} attack recovers all secret primes
$\{p_i\}_i$ in polynomial time given the set $\mathcal{A}$
of level-one encodings of zero 
and the sets $\mathcal{B}$ and $\mathcal{C}$.\\

\subsubsection{Attacking CLT13 with fewer encodings}
\label{ss:improve_clt13}

We consider the following CLT13-based problem.

\begin{defi}[CLT13 Problem]\label{def:clt13pb}
	Let $n \geq 2$ be the dimension of {\upshape CLT13} 
	and $x_0=\prod_{i=1}^n p_i$.
	Let $\mathcal{E}$ be a finite non-empty set of encodings at level $1$
	and $\mathcal{E}_0 \subseteq \mathcal{E}$ a non-empty subset
	such that every element of $\mathcal{E}_0$ is an encoding of zero.
	The {\upshape CLT13} Problem is as follows:
	Given the sets $\mathcal{E}$ and $\mathcal{E}_0$, factor $x_0$.
\end{defi}

We refer to $\mathcal{E}$ and $\mathcal{E}_0$ as the sets of 
"available encodings" and "available encodings of zero", respectively.
It is not a loss of generality to consider level-one encodings.
As in \cite{cheon15}, we write
$\mathcal{E}=\mathcal{A} \cup \mathcal{B} \cup \mathcal{C}$ 
with $\mathcal{A} \subseteq \mathcal{E}_0$.
As recalled above, 
\cite{cheon15} requires $\#\mathcal{E}_0 \geq n$
to factor $x_0$, and a total number of public encodings
$\#\mathcal{E} = 2n + 2$.

We aim at reducing the number of encodings 
needed for the factorization of $x_0$
and treat the following questions independently: 

\begin{enumerate}
	\item [(i)]
	Factor $x_0$ with fewer  
	available encodings of zero,
	i.e.  $\#\mathcal{E}_0 < n$
	
	\item [(ii)]
	Factor $x_0$ with fewer available encodings, i.e.
	$\#\mathcal{E} < 2n+2$\\
\end{enumerate}

\paragraph{\itshape{A naive improvement.}}
As for the CRT-ACD Problem, 
there is a naive improvement using fewer encodings, but 
assuming $\kappa = 4$.
One can form product encodings 
$\alpha_j \beta_a \gamma_k \delta_{\ell}$ at level $4$,
where every encoding
is at level $1$.
These can be partitioned into sets 
$\mathcal{A}, \mathcal{B}$ and $\mathcal{C}$ 
such that $\mathcal{A}$ corresponds to
encodings of zero 
with $\#\mathcal{A}= \mathcal{O}(\sqrt{n})$.
However, this approach has the inconvenience 
of using $\kappa=4$
and 
our improved attack aims at lowering 
the number of public encodings 
while $\kappa = 3$. \\

\paragraph{\itshape{Minimizing the number of encodings of zero}}
We explain how to use Algorithm $\cA_{\CC}$ 
to factor $x_0$ using only $\#\cE_0=\cO(\sqrt{n})$ level-one 
encodings of zero.\\

We fix integers $2 \leq p < n$ and $3 \leq t < n$
and assume again $\kappa=3$.
As in \cite{cheon15}, we write 
$\cE=\cA \cup \cB \cup \cC$ 
with $\cA \subseteq \cE_0$.
We let $\#\cA = p$, $\#\cB = t$ and $\#\cC = n$; 
and claim $p = \cO(\sqrt{n})$.

Every product encoding $c = \alpha_j \beta_a \gamma_k$ with $(\alpha_j,\beta_a,\gamma_k) \in \cA \times \cB \times \cC$
is an encoding of zero
and by correct zero-testing we obtain 
integer matrix relations 
\begin{equation}\label{eq:matrix_improved_clt13}
W_a = P \cdot U_a \cdot Q \quad , \quad a \in [t]
\end{equation}  
for  
$P \in \ZZ^{p \times n}, 
Q \in \ZZ^{n \times n}$
corresponding to encodings 
in $\cA$ and $\cC$, respectively,
and diagonal matrices $\{U_a\}_a$ corresponding to $\cB$.
Exactly as in \cite{cheon15}, the matrices $\{U_a\}_a$
contain integers $\beta_{ai}$ such that 
$\beta_a \equiv \beta_{ai} \pmod{p_i}$
for $i \in [n]$.
With high probability the ranks of $P$ and $Q$ are $p$ and $n$, respectively.
Defining $W_0' = W_1$ and $W_a' = W_{a-1}$ for 
$2 \leq a \leq t$ we obtain an
instance similar to Problem
$\CC$ of Def.~\ref{def:prABCD},
but without a "special input matrix" $PQ$ (see Sec.~\ref{ss:remarks_defABCD}).
Using Algorithm $\cA_\CC$, we 
reveal eigenvalues (the diagonal entries) 
of the matrices $\{U_a U_1^{-1}\}_a$
as it is likely that $U_1$ be invertible.
We finally deduce the prime factorization of $x_0$ 
by taking greatest common divisors, as in \cite{cheon15}.

By the optimization in Sec.~\ref{ss:optimize_C},
we choose 
$t = \lceil \sqrt{2n} \rceil$ and
 $\#\cA=p=\lceil\sqrt{2n}\rceil$.
\\
\paragraph{\itshape{Minimizing the total number of encodings}}
We now explain how to use Algorithm 
$\cA_\DD$ to factor $x_0$ using 
$\#\cE = \frac{4}{3}n + \cO(\sqrt{n})$
instead of $\#\cE = 2n+2$ as in \cite{cheon15}.\\

Contrary to the previous case, we 
now use a set $\cC$ with $\#\cC = p$;
so $\#\cE = 2p+t$.
It is now direct to see 
that upon correct zero-testing
we derive equations as in \eqref{eq:matrix_improved_clt13}
but with $Q \in \ZZ^{n \times p}$ instead.
Thus, if both $P$ and $Q$ have rank $p$,
we obtain Problem $\DD$ of Def.~\ref{def:prABCD} without 
"special input matrix" $W_0$.
Then Algorithm $\cA_\DD$ reveals
ratios of diagonal entries of $\{U_a U_1^{-1}\}$
and we consequently factor $x_0$.

Following Sec.~\ref{ss:optimize_D}, we are led 
to minimize $\#\cE(n) = 2p+t$ as a function of $n$. 
We can let  
$
p = \lceil \frac{2}{3}n + \frac{1}{3\sqrt{2}}\sqrt{n} \rceil $
and
$
t = \lceil \frac{1}{3}\sqrt{2n} + \frac{5}{3} \rceil
$
and obtain
$$
\#\mathcal{E}(n) = 2\Bigl\lceil \frac{2}{3}n + 
\frac{1}{3\sqrt{2}}\sqrt{n} \Bigr\rceil + \Bigl\lceil \frac{1}{3}\sqrt{2n} + \frac{5}{3} \Bigr\rceil = \frac{4}{3}n + \mathcal{O}(\sqrt{n}) \ .
$$

\paragraph{\itshape{Cryptanalysis with independent slots}}
In \cite{cltindslots}, Coron and Notarnicola
cryptanalyze CLT13 when no encodings of zero are available beforehand, 
but instead only "partial-zero" encodings.
Messages are non-zero modulo a product 
of several primes $g_1 \cdots g_{\theta}$
for some integer $\theta \in [n]$.
We can 
improve this cryptanalysis following the same techniques as above.
Let $\ell$ the number of partial-zero encodings.
Since \cite{cltindslots} is based on the algorithm of 
Cheon {\sl et al.}~to factor $x_0$,
we can now replace it by 
Algorithm $\cA_\CC$ once 
$\ell$ encodings of zero are created.
This means that we can set $\ell = \cO(\sqrt{n})$,
which brings a twofold improvement:
first, lattice reduction (in the orthogonal lattice attack \cite[Sec. 4]{cltindslots}) 
is only run on a lattice of dimension $\cO(\sqrt{n})$;
and second, 
the number of partial-zero encodings is reduced to 
$\cO(\sqrt{n})$.

\section{Computational Aspects and Practical Results}
\label{s:practical}

We describe practical parameters for algorithms 
$\cA_\CC$ and $\cA_\DD$.
We have implemented our algorithms in SageMath \cite{sage}; 
our source code is provided in
\url{https://pastebin.com/Yg6QgZTh}. 
Our experiments are done on a standard $3,3$ GHz Intel Core i7 processor.

\subsection{Instance Generation of Problems $\CC$ and $\DD$}

As is the case for applications in cryptanalysis, we consider matrices with integer entries.
To generate instances of Problems $\CC$ and $\DD$,
given fixed parameters $n,t,p$,
we uniformly at random generate matrices 
$P,Q$ and $\{U_a\}_a$ with 
entries in $[-k,k]\cap \ZZ$ for some $k \in \ZZ_{\geq 1}$ as in 
Def.~\ref{def:prABCD}.
We set $W_0=PQ$ and $W_a=PU_aQ$ for $a \in [t]$
to obtain instances of Problems $\CC$ or $\DD$.

We perform the linear algebra 
over $\ZZ/\ell\ZZ$ for a large prime $\ell$,
instead of over $\QQ$.
It suffices to choose $\ell$ slightly larger than the diagonal 
entries of $\{U_a\}_a$ 
(e.g. $\ell = \cO(\max_{a,i}|u_{ai}|)$, 
where $u_{ai}$ for $i \in [n]$ 
denote the diagonal entries of $U_a$).
The running time depends on the entry size of the
generated matrices.
The overall computational cost of our algorithms $\mathcal{A}_{\CC}$ and 
$\mathcal{A}_{\DD}$ 
is dominated by the 
cost of solving systems of linear equations and performing 
simultaneous diagonalization,
which can be done by standard algorithms for non-sparse linear algebra.

\subsection{Practical Experiments}

We gather practical parameters for 
problems $\CC$ and $\DD$, 
and for our applications of Sec.~\ref{s:applications}.
We compare $p,t$ with the theoretical 
values $p_0(n),t_0(n)$ obtained in Sec.~\ref{s:algo_C} and \ref{s:algo_D}. 
For Table 1, $p_0(n) = \lceil{\sqrt{2n}}\rceil$
and $t_0(n) = \lceil \sqrt{2n}-1\rceil$.
For Table 2, 
$p_0(n)=\lceil \frac{2}{3}n + \frac{\sqrt{n}}{3\sqrt{2}}\rceil$
and $t_0(n) = \lceil \frac{1}{3}(\sqrt{2n}+5)\rceil$.
Here "entry size" is an approximation of the bit-size of the 
max-norm of each input matrix. 

In Table 3, we compare our work with \cite{corper} for the 
CRT-ACD Problem
and with \cite{cheon15} for the cryptanalysis of CLT13.
We give parameters for obtaining 
a complete factorization
of $M$ (in CRT-ACD) and $x_0$ (in CLT13) of approximate bit-size $n\eta$.
For CRT-ACD, the column "this work" equals $\#\cS=p+t$ (Series~1).
For CLT13, "this work" shows $\#\cE = 2p+t$ (Series~2)
and $\#\cE_0 = p$ (Series~3).
For example, for $n=50$, our algorithm factors $M$ (in CRT-ACD) 
using only $19$ public samples, whereas \cite{corper} requires $51$ samples; 
and similarly breaks CLT13 with only 
$10$ public encodings of zero, while \cite{cheon15} uses $50$.

In conclusion, these practical experiments overall confirm our theory, as well as
the quadratic improvement over \cite{corper} and \cite{cheon15}.

\begin{table}\label{tab:practical_C}
\begin{center}
	\begin{tabular}{|c||c||c|c||c|c||c|}\cline{3-6}
		\multicolumn{2}{c|}{}& \multicolumn{2}{|c|}{~Practice~} &
		\multicolumn{2}{|c|}{~Theory~} & \multicolumn{1}{c}{}\\ \hline
		~$n$~ 	&~Entry size~ & 	\   ~$p$~  \       &   
		~$t$~      &       \  ~$p_0(n)$~  \     &    
		~$t_0(n)$~    &   \multicolumn{1}{|c|}{~Running time~}    \\ \hline
		~$15$~ &~$1000$~&        ~$6$~        &      ~$4$~         &       ~$6$~         &     ~$5$~ &  $4$ min $4$ s   \\\hline
		~$25$~ &~$750$~&        ~$8$~        &      ~$7$~         &       ~$8$~         &     ~$7$~ &  $3$ min $45$ s    \\\hline
		~$50$~ &~$600$~&        ~$10$~        &      ~$9$~         &       ~$10$~         &  ~$9$~ &  $4$ min $34$ s    \\\hline
		~$100$~ &~$200$~&        ~$15$~        &      ~$14$~         &       ~$15$~      &    ~$14$~ &  $1$ h $17$ min    \\\hline
		~$150$~ &~$100$~&        ~$18$~        &      ~$16$~         &       ~$18$~       &      ~$17$~ &  $6$ h $29$ min    \\\hline
		~$500$~ & ~$20$~ &        ~$32$~        &      ~$31$~         &       ~$32$~       &    ~$31$~ &  $29$ min $3$ s    \\\hline
	\end{tabular}
\end{center}
\caption{Experimental data for Algorithm $\cA_\CC$}
\end{table}

\begin{table}\label{tab:practical_D}
\begin{center}
	\begin{tabular}{|c||c||c|c||c|c||c|}\cline{3-6}
		\multicolumn{2}{c|}{}& \multicolumn{2}{|c|}{~Practice~} &
		\multicolumn{2}{|c|}{~Theory~} & \multicolumn{1}{c}{}\\ \hline
		~$n$~ 	&~Entry size~ & 	\ ~$p$~ \      &    ~$t$~      &       \ ~$p_0(n)$~  \     &    
		 ~$t_0(n)$~    &   \multicolumn{1}{|c|}{~Running time~}    \\ \hline
		~$15$~ &~$1000$~&        ~$11$~        &      ~$4$~         &       ~$11$~         &     ~$4$~ &  $4$ min $2$ s  \\\hline
		~$25$~ &~$750$~&        ~$18$~        &      ~$4$~         &       ~$18$~         &     ~$5$~ &   $1$ min $54$ s  \\\hline
		~$50$~ &~$600$~&        ~$35$~        &      ~$5$~         &       ~$35$~         &  ~$5$~ &  $1$ min $39$ s  \\\hline
		~$100$~ &~$200$~&        ~$70$~        &      ~$7$~         &       ~$70$~      &    ~$7$~ &    $5$ min $14$ s   \\\hline
		~$150$~ &~$100$~&        ~$103$~        &      ~$8$~         &       ~$103$~       &      ~$8$~ &  $23$ min $14$ s  \\\hline
		~$500$~ & ~$20$~&        ~$339$~        &      ~$13$~         &       ~$339$~       &  ~$13$~  & ~$6$ min $57$ s~     \\\hline
	\end{tabular}
\end{center}
\caption{Experimental data for Algorithm $\cA_\DD$}
\end{table}

\begin{table}
\begin{center}
 	\begin{tabular}{|c||c|c||c|c|c|c|}\cline{6-7}
 			\multicolumn{5}{c|}{Series 1}&
 			\multicolumn{2}{|c|}{~Num. of samples~} \\ 
 			\hline
 			~$n$~ 	& 	\ \ \ ~$\eta$~ \ \ \      &    ~$\rho$~      &       \ \  ~$p$~  \ \     &     ~$t$~    &       ~this work~    & \  ~\cite{corper}~ \ \\   \hline
 			~$20$~ &        ~$1000$~        &      ~$200$~         &       ~$7$~         &     ~$6$~ & ~$13$~  & ~$21$~ \\ \hline
 			~$30$~ &        ~$1000$~        &      ~$100$~         &       ~$8$~         &     ~$7$~ &  ~$15$~   &~$31$~ \\ \hline
 			~$50$~ &        ~$800$~        &      ~$100$~         &       ~$10$~         &  ~$9$~ &  ~$19$~ &~$51$~ \\ \hline
 			\multicolumn{5}{c|}{Series 2}&
 			\multicolumn{2}{|c|}{ ~Num. of encodings ~ } \\ 
 			\hline
 			~$n$~ 	& 	\ \ \ ~$\eta$~ \ \ \      &    ~$\rho$~      &  ~$p$~  &     ~$t$~    & 
 			~this work~    & \  ~\cite{cheon15}~ \ \\   \hline
 			~$20$~ &        ~$1000$~        &      ~$200$~         &       ~$15$~         &     ~$4$~ & ~$34$~  & ~$42$~ \\ \hline
 			~$30$~ &        ~$1000$~        &      ~$100$~         &       ~$22$~         &     ~$5$~ &  ~$49$~   &~$62$~ \\ \hline
 			~$50$~ &        ~$800$~        &      ~$100$~         &       ~$35$~         &  ~$5$~ &  ~$75$~ &~$102$~ \\ \hline
 			\multicolumn{5}{c|}{Series 3}&
 			\multicolumn{2}{|c|}{ ~Num. of encodings of zero~ } \\ 
 			\hline
 			~$n$~ 	& 	\ \ \ ~$\eta$~ \ \ \      &    ~$\rho$~      &    ~$p$~   &     ~$t$~    &    \  ~this work~  \ &  ~\cite{cheon15}~  \\   \hline
 			~$20$~ &        ~$1000$~        &      ~$200$~         &       ~$7$~         &     ~$6$~ & ~$7$~  & ~$20$~ \\ \hline
 			~$30$~ &        ~$1000$~        &      ~$100$~         &       ~$8$~         &     ~$7$~ &  ~$8$~   &~$30$~ \\ \hline
 			~$50$~ &        ~$800$~        &      ~$100$~         &       ~$10$~         &  ~$9$~ &  ~$10$~ &~$50$~ \\ \hline
 	\end{tabular}
 \end{center}
\caption{Experimental data for the CRT-ACD Problem
and the CLT13 Problem}
\end{table}

\subsection*{Acknowledgments}
The authors thank the anonymous reviewers of ANTS-XIV for their helpful comments.
The second author acknowledges support by the Luxembourg National Research Fund through grant PRIDE15/10621687/- SPsquared.
 
\bibliographystyle{alpha}
\bibliography{CLTBiblio}

\end{document}